\documentclass[10pt]{amsart}

\usepackage{fancyhdr}
\usepackage{stmaryrd}
\usepackage{calrsfs}

\usepackage{amsmath}
\usepackage[nospace,noadjust]{cite}
\usepackage{amsfonts}
\usepackage{amssymb,enumerate}
\usepackage{amsthm}
\usepackage{cite}
\usepackage{comment}
\usepackage{color}
\usepackage[all]{xy}
\usepackage{hyperref}
\usepackage{lineno}

\newtheorem*{maintheorem*}{Main Theorem}
\newtheorem{theorem}{Theorem}[section]
\newtheorem{prop}[theorem]{Proposition}

\newtheorem{lemma}[theorem]{Lemma}

\theoremstyle{definition}

\newtheorem{remark}[theorem]{Remark}
\newtheorem{example}[theorem]{Example}
\numberwithin{equation}{section}

\newcommand{\nn}{\mathbb{N}}
\newcommand{\pp}{\mathbb{P}}
\newcommand{\qq}{\mathbb{Q}}
\newcommand{\rr}{\mathbb{R}}
\newcommand{\zz}{\mathbb{Z}}

\newcommand{\red}{{\text{\rm red}}}

\providecommand\ldb{\llbracket}
\providecommand\rdb{\rrbracket}

\hypersetup{pdftitle={FD}, colorlinks=true, linkcolor=blue, citecolor=cyan, urlcolor=wine}

\begin{document}
	
\mbox{}
\title{On monoid algebras having \\ every nonempty subset of $\nn_{\ge 2}$ as a length set}

\author{Alfred Geroldinger}
\address{Department of Mathematics and Scientific Computing\\ University of Graz,  Heinrichstr. 36\\ 8010 Graz, Austria}
\email{alfred.geroldinger@uni-graz.at}
\urladdr{https://imsc.uni-graz.at/geroldinger}

\author{Felix Gotti}
\address{Department of Mathematics\\Massachusetts Institute of Technology\\Cambridge, MA 02139, USA}
\email{fgotti@mit.edu}
\urladdr{https://felixgotti.com}


\thanks{While preparing this paper, the first author was supported by the Austrian Science Fund FWF (Project P36852-N), while the second author was supported by the NSF award DMS-2213323.}


\begin{abstract}
	We construct monoid algebras which satisfy the ascending chain condition on principal ideals and which have the property that every nonempty subset of $\mathbb N_{\ge 2}$ occurs as a length set.
\end{abstract}
\medskip

\maketitle

\bigskip
%%%%%%%%%%%
%%%%%%%%%%%
\section{Introduction}
\label{sec:intro}
\smallskip

Let $D$ be a (commutative integral) domain or a (commutative cancellative) monoid. If $D$ satisfies the ascending chain condition on principal ideals (ACCP), then $D$ is atomic, that is, every non-zero non-invertible element of $D$  can be written as a product of atoms. If $b = a_1 \cdots a_\ell \in D$, where $\ell \in \mathbb N$ and $a_1, \ldots, a_\ell$ are atoms of~$D$, then $\ell$ is called a factorization length of $b$ and the set $\mathsf L (b) \subseteq \mathbb N$ of all factorization lengths is called the length set of $b$. If all length sets of $D$ are  finite and nonempty, then $D$ is called a BF-domain resp. a BF-monoid, in which case we say that $D$ satisfies the bounded factorization property. It is well known that a domain/monoid satisfies the bounded factorization property provided that it satisfies the ascending chain condition on divisorial ideals, which holds for all Noetherian domains and all Krull monoids, and so all Krull domains. It is also well known that every domain/monoid satisfying the bounded factorization property must satisfy the ACCP.
\smallskip

The system $\mathcal L (D) = \{ \mathsf L (b) : b \in D\}$ consisting of all length sets of $D$, together with invariants controlling their structure, is a key arithmetic invariant describing the non-uniqueness of factorizations in $D$. There is a rich literature on length sets. However, so far the study of length sets has been taking place almost entirely on BF-domains/monoids. One type of results states that, under reasonable algebraic finiteness conditions (e.g., the finiteness of the class group in case of Krull monoids), length sets are highly structured (see the monograph \cite[Chapter 4]{Ge-HK06a} and the surveys \cite{Sc16a, Ge-Zh20a}, and see \cite{B-B-N-S23a} for progress in noncommutative rings). A second type of results reveals settings in which every  finite, nonempty subset of $\mathbb N_{\ge 2}$ occurs as a length set: these settings include Krull monoids with infinite class group, rings of integer-valued polynomials, and more (see \cite[Theorem 7.4.1]{Ge-HK06a}, \cite{fG20,Go19a, Fa-Zh23a, Fa-Wi24a}). In both settings the arithmetic of monoid algebras has found special attention (see \cite{Ki01b, Co-Go19a,  Go22a, Ch-Fa-Wi22a, Fa-Wi22a, Fa-Wi22b, Fa-Wi22c}).
\smallskip

Sporadic examples of atomic domains that do not satisfy the ACCP appeared in the literature a few decades ago (see~\cite{aG74,mR93,aZ82}). However, it was not until the last few years that detailed investigations were carried out to understand the subtle difference between atomic domains and domains satisfying the ACCP, or domains satisfying the bounded factorization property (see \cite{Tr23a}). For instance, in~\cite{GL23}, Grams' construction of the first atomic domain not satisfying the ACCP was significantly generalized, while in~\cite{GL23a} a weaker notion of the ACCP was introduced with the purpose to identify several classes of atomic domains not satisfying the ACCP. For more in this direction, see the recent survey \cite{Co-Go25a}. Besides the algebraic objects discussed in the survey, further constructions of atomic monoids, domains, and rings not satisfying the ACCP have been recently provided in~\cite{BC19,GL22,GL23a,B-B-N-S23a}.
\smallskip

In~\cite{CGG20a}, the arithmetic of factorizations of the rank-one additive monoids $\nn_0[q]$ (for positive rationals $q$) was investigated. Among other results in the same paper, the system of length sets of $\nn_0[q]$ is determined for every $q$ such that $\nn_0[q]$ is atomic. When $q \in (0,1)$ and $q^{-1} \notin \nn$, one can check that $\nn_0[q]$ is atomic but does not satisfy the ACCP and, in this case, the non-singleton length sets of $\nn_0[q]$ are infinite arithmetic progressions \cite[Theorem~3.3]{CGG20a}. Besides the computation in this very special monoid (carried out in~\cite{CGG20a}), no systematic study of infinite length sets in non-BF-domains/monoids has been provided so far. The present paper should be seen as a starting point in this direction and we start with an extreme case: we construct  integral domains whose systems of length sets are as large as they can possibly be. More precisely, we show that there are monoid algebras which satisfy the ACCP (but not the bounded factorization property) such that every nonempty subset of $\mathbb N_{\ge 2}$ occurs as a length set. We proceed to formulate the main result of this paper.
\smallskip

\begin{theorem} \label{1.1}
	Let $D$ be an integral domain that satisfies the {\rm ACCP}. Then there exists a positive submonoid~$M$ of a $\mathbb Q$-vector space such that the monoid algebra $D[M]$ satisfies the {\rm ACCP} and every nonempty subset of $\nn_{\ge 2}$ occurs as a length set of $D[M]$.
\end{theorem}
\smallskip

In Section \ref{sec:background}, we gather some preliminaries on factorizations in monoids and domains. In Section \ref{sec:overatomicity}, we construct, for any prescribed nonempty subset $L \subseteq \mathbb N_{\ge 2}$, a Puiseux monoid (i.e., an additive submonoid of~$\mathbb Q_{\ge 0})$ that satisfies the ACCP and which has $L$ as a length set (Theorem \ref{3.5}). The construction is based on a realization result for length sets in numerical monoids (Proposition \ref{realization-finite-set}) and on the properties of a very special Puiseux monoid (Proposition \ref{prop:prime reciprocal}). For this Puiseux monoid, we also determine the system of length sets, and this is the first example of a non-BF-monoid for which the system of length sets is fully determined. On our way, we also study how the ACCP property behaves under various algebraic constructions (Proposition~\ref{prop:ACCP}). The proof of our main theorem, Theorem \ref{1.1}, is provided in Section~\ref{4}.

\bigskip
%%%%%%%%%%%%
%%%%%%%%%%%%
\section{Background}  
\label{sec:background}

\smallskip
%%%%%%%%%%%%%%%%
\subsection{General Notation} 

We denote by $\mathbb P, \mathbb N, \mathbb N_0, \mathbb Z, \mathbb Q$, and $\mathbb R$ the set of prime numbers, positive integers, non-negative integers, integers, rational integers, and real numbers, respectively. For a subset $S$ of $\rr$, we set $S^\bullet := S \setminus \{0\}$. For $b,c \in \mathbb R$, we let
\[
	\ldb b,c \rdb := \{ k \in \mathbb Z : b \le k \le c \}
\]
denote the discrete interval between $b$ and $c$. Let $A$ and $B$ be subsets of an additive abelian group $G$. Then $A+B = \{a+b : a \in A \text{ and } b \in B\} \subseteq G$ denotes their sumset. For $a \in G$, $a + B = \{a\}+B$,  and, for $n \in \mathbb N_0$,
\[
	nA := \underbrace{A + \dots + A}_\text{$n$ times}
\]
means the $n$-fold sumset, where $nA = \{0\}$ when $n=0$. For $q \in \mathbb Q^\bullet$, the unique integers $n$ and~$d$ with $d > 0$ such that $q = n/d$, and $\gcd(n,d)=1$ are denoted by $\mathsf{n}(q)$ and $\mathsf{d}(q)$, respectively. For any subset $Q$ of $\qq$ consisting of non-zero rationals, we set
\[
	\mathsf{d}(Q) := \{\mathsf{d}(q) : q \in Q\}.
\]
For $p \in \pp$ and $n \in \zz^\bullet$, we set $\mathsf v_p(n) := \max \{m \in \nn_0 : p^m \mid n\}$. Then for any prime $p \in \mathbb P$, we let $\mathsf v_p \colon \mathbb Q^{\bullet} \to \mathbb Z$ denote the $p$-adic valuation map: $\mathsf v_p(q) = \mathsf v_p(\mathsf{n}(q)) - \mathsf v_p(\mathsf{d}(q))$ for all $q \in \qq^\bullet$.

\medskip
%%%%%%%%%%%%%%%%%
\subsection{Monoids and Ideals} 

By a {\it monoid}, we mean a cancellative and commutative semigroup with an identity element. We will use both additive notation (for submonoids of $\mathbb Q$-vector spaces) and multiplicative notation (for monoids of non-zero elements of integral domains). We briefly gather some key concepts of factorization theory we shall be using later. For this we use  multiplicative notation for monoids, and note that all the notions and results we mention in the setting of monoids naturally adapt to the setting of integral domains by means of their multiplicative monoids.

For a set $P$, we let $\mathcal F (P)$ denote the free commutative monoid with basis $P$. An element $b \in \mathcal F (P)$ will be written in the form
\[
	b = \prod_{p \in P} p^{\mathsf v_p (b)},
\]
where, for each $p \in P$, the map $\mathsf v_p \colon \mathcal F (P) \to \mathbb N_0$ is the $p$-adic valuation map on $\mathcal F(P)$. Furthermore, we call
\[
	|b| = \sum_{p \in P} \mathsf v_p (b) \in \mathbb N_0
\]
the {\it length} of $b$. Let $M$ be a (multiplicatively written) monoid with identity element $1 := 1_M$. We denote by $M^{\times}$ the group of units of $M$ and by $\mathsf q (M)$ the quotient group of $M$. The monoid $M$ is called {\it reduced} if $M^{\times} = \{1\}$, while $M$ is called {\it torsion-free} if $\mathsf q (M)$ is a torsion-free abelian group. The \emph{rank} of $M$ is the dimension of the $\qq$-vector space $\qq \otimes_\zz \mathsf{q}(M)$. The monoid $M_{\red} := M/M^{\times}$ is called the \emph{reduced} monoid of~$M$. For a nonempty set of indices $I$, the \emph{coproduct} of a family of monoids $(M_i)_{i \in I}$ is
\[
	\coprod_{i \in I}M_i := \big\{ (b_i)_{i \in I} : b_i=1 \ \text{for almost all} \ i \in I \big\}, %\ \subseteq \ \prod_{i \in I} M_i,
\]
which is a submonoid of $\prod_{i \in I} M_i$ and, therefore, a monoid. If each $M_i$ is reduced, then the coproduct of $(M_i)_{i \in I}$ is also reduced. A submonoid $S \subseteq M$ is said to be {\it divisor-closed} if for all $b \in S$ and $c \in M$ the divisibility relation $c \mid_M b$ implies that $c \in S$.

For a subset $S$ of $M$, we let $\langle S \rangle$ denote the submonoid of $M$ generated by $S$. The monoid $M$ is called \emph{finitely generated} provided that $M = \langle S \rangle$ for a finite subset $S$ of $M$. A subset $I \subseteq M$ is called an $s$-\emph{ideal} if $IM = I$. The monoid $M$ is $s$-\emph{noetherian} if every ascending chain of $s$-ideals stabilizes. A reduced monoid is finitely generated if and only if it is $s$-noetherian (\cite[Theorem 3.6]{HK98}). For a subset $S$ of $\mathsf{q}(M)$, we set $S^{-1} := \{x \in \mathsf{q}(M) : xS \subseteq M \}$ and $S_v := (S^{-1})^{-1}$. An $s$-ideal $I$ of $M$ is called a \emph{v-ideal} if $I_v = I$.\footnote{$v$-ideals are also called divisorial ideals.} We say that $M$ is $v$-{\it noetherian} (or a {\it Mori monoid}) if every ascending chain of $v$-ideals stabilizes.  An $s$-ideal of the form $bM$ for some $b \in M$ is called \emph{principal}, and $M$ is said to satisfy the \emph{ascending chain condition on principal ideals} (ACCP) if every ascending chain of principal ideals stabilizes.

Two classes of monoids will play a crucial role in this paper: these are numerical monoids and Puiseux monoids. A {\it numerical monoid} is a submonoid $M$ of $(\nn_0, +)$ such that the complement $\nn_0 \setminus M$ of $M$ in $\nn_0$ is finite. For a numerical monoid $M$, we denote by $\mathsf{f}(M) = \max (\nn_0 \setminus M)$  the {\it Frobenius number} of $M$. One can readily check that every numerical monoid is reduced and finitely generated and so $s$-noetherian. A {\it Puiseux monoid} is a submonoid of $(\mathbb Q_{\ge 0}, +)$. We say that a Puiseux monoid $M$ is \emph{nontrivial} if $M \ne \{0\}$. Every numerical monoid is clearly a Puiseux monoid. On the other hand, every finitely generated Puiseux monoid is isomorphic to a numerical monoid. Furthermore, every monoid that is not a group and whose quotient group is a rank-one torsion-free abelian group is isomorphic to a Puiseux monoid (\cite[Theorem 3.12]{Ge-Go-Tr21}).

\medskip
%%%%%%%%%%%%%%%
\subsection{Factorizations} 

A non-unit $a \in M$ is called an \emph{atom} if the equality $a = bc$ for some $b,c \in M$ ensures that $b \in M^\times$ or $c \in M^\times$. We denote by $\mathcal A (M)$ the set of atoms of $M$. Then we let $\mathsf Z (M)$ denote the free commutative monoid on the set $\mathcal A (M_{\red})$, and the formal products in $\mathsf{Z}(M)$ are called \emph{factorizations}. Since $\mathsf{Z}(M)$ is a free commutative monoid, there exists a unique monoid homomorphism $\pi \colon \mathsf Z (M) \to M_{\red}$ such that $\pi(a) = a$ for all $a \in \mathcal{A}(M_{\text{red}})$. For an element $b \in M$, we set
\[
	\mathsf Z (b) := \pi^{-1} (bM^\times).
\]
If a factorization $z$ in $\mathsf{Z}(M)$ is the formal product of $\ell$ atoms of $\mathsf{Z}(M)$ (counting repetitions), then we say that the \emph{length} of $z$ is $\ell$ and we set $|z| := \ell$. For each $b \in M$, the subset
\[
	\mathsf L(b) := \mathsf L_M (b) := \{ |z| : z \in \mathsf Z (b) \}
\]
of $\mathbb N_0$ consisting of all the factorization lengths of $b$ is called the \emph{length set} of $b$. For each $b \in M$, note that $\mathsf L (b) = \{0\}$ if and only if  $b \in M^{\times}$, while $1 \in \mathsf L (b)$ if and only if $\mathsf L (b) = \{1\}$, which happens precisely when~$b$ is an atom. We say that an element $b \in M$ has {\it unique factorization} if $|\mathsf Z (b)| = 1$. If $S$ is a divisor-closed submonoid of $M$, then $\mathsf L_S (b) = \mathsf L_M (b)$ for every $b \in S$. The set
\[
	\mathcal L (M) := \{\mathsf L (b) : b \in M \}
\]
is called the {\it system of length sets} of $M$. We say that $M$ is
\begin{itemize}
	\item {\it factorial} if every $b \in M$ has unique factorization;
	\smallskip
	
	\item a BF-{\it monoid} (or a \emph{bounded factorization monoid}) if $\mathsf L (b)$ is finite and nonempty for all $b \in M$;
	\smallskip
	
	\item {\it atomic} if $\mathsf L (b)$ is nonempty for all $b \in M$.
\end{itemize}
We summarize some basic properties which will be used without further mention. Every Mori monoid (whence every Krull monoid, every Krull domain, and every Noetherian domain) is a BF-monoid \cite[Theorem~2.2.9]{Ge-HK06a}, but not every BF-domain is a Mori domain \cite[Example~4.7]{Co-Go25a}. In addition, every BF-monoid satisfies the ACCP \cite[Corollary~1.3.3]{Ge-HK06a}, but not every integral domain that satisfies the ACCP is a BF-domain \cite[Example~2.1]{AAZ90}. Finally, every monoid that satisfies the ACCP is atomic, but not every atomic domain satisfies the ACCP \cite[Theorem~1.3]{aG74}.

\medskip
%%%%%%%%%%%%%%%%%%%%%
\subsection{Some Preliminary Results} 

We continue with two simple lemmas which, most probably, are well known. We provide short proofs here for the convenience of the readers.

\begin{lemma} \label{lem:atoms of the union}
	Let $M$ be a monoid. For a nonempty set of indices $I$, let  $(M_i)_{i \in I}$ be a family of submonoids of $M$ such that $M = \bigcup_{i \in I} M_i$. If $M_i^{\times} = M^{\times} \cap M_i$ for all $i \in I$, then $\mathcal{A}(M) \subseteq \bigcup_{i \in I} \mathcal{A}(M_i)$.
\end{lemma}

\begin{proof}
	Let $a \in \mathcal A (M)$. Then there is $i \in I$ with $a \in M_i$. If $b, c \in M_i$ with $a = bc$, then $b \in M^{\times}$ or $c \in M^{\times}$. Thus, $b \in M_i \cap M^{\times} = M_i^{\times}$ or $c \in M_i \cap M^{\times} = M_i^{\times}$, whence $a \in \mathcal A (M_i)$.
\end{proof}

\smallskip
\begin{lemma} \label{2.2}
	Let $M$ be a monoid satisfying the {\rm ACCP}.
	\begin{enumerate}
		\item If $S \subseteq M$ is a  submonoid with $S^{\times} = M^{\times} \cap S$, then $S$ satisfies the {\rm ACCP}.
		\smallskip
		
		\item If $(M_i)_{i \in I}$ is a family of monoids satisfying the {\rm ACCP}, then their coproduct satisfies the {\rm ACCP}.
	\end{enumerate}
\end{lemma}

\begin{proof}
	(1) Let $(b_nS)_{n \ge 1}$ be an ascending chain of principal ideals of $S$ with $b_n \in S$ for all $n \ge 1$. Then $(b_nM)_{n \ge 1}$ is an ascending chain of principal ideals of $M$, whence it stabilizes and so we can assume that $b_nM = b_mM$ for some $m \in \mathbb N$ and all $n \ge m$. Fix $n \in \nn$ with $n \ge m$, and write $b_m = b_n u$ for some $u \in M^{\times}$. Since $b_mS \subseteq b_nS$, it follows that $u \in M^{\times} \cap S = S^{\times}$, whence $b_mS = b_n S$.
	\smallskip
	
	(2) Let $\big( (b_{i,n}M_i)_{i \in I} \big)_{n\ge 1}$ be an ascending chain of principal ideals in the coproduct $\coprod_{i \in I} M_i$. Then $(b_{i,n}M_i)_{n \ge 1}$ is an ascending chain of principal ideals in $M_i$ for all $i \in I$. Since only finitely many $b_{i,1}$ are distinct from $1_{M_i}$, there are only finitely many nontrivial chains. Thus, the chain $\big( (b_{i,n}M_i)_{i \in I} \big)_{n \ge 1}$ stabilizes in the coproduct.
\end{proof}

Let $M_1, M_2$ be monoids satisfying the ACCP. Then part~(2) of Lemma~\ref{2.2} implies that their direct product $M_1 \times M_2$ also satisfies the ACCP. However, if $M$ is a monoid containing $M_1$ and $M_2$ as submonoids with $M = M_1M_2$, then $M$ does not necessarily satisfy the ACCP (see Example \ref{3.4}). We will obtain a positive result in Section~\ref{sec:overatomicity} under some additional assumptions (Proposition~\ref{prop:ACCP}). We start with some preparation, which consists of the following two lemmas.

\smallskip
\begin{lemma}\label{2.3}
	Let $N$ be an $s$-noetherian monoid, and let $(s_n)_{n \ge 1}$ be a sequence of elements of $N$. Then there exists a map $\varphi:\mathbb{N}\rightarrow\mathbb{N}$ such that $\varphi(n+1)>\varphi(n)$ and $s_{\varphi(n+1)}N\subseteq s_{\varphi(n)}N$ for every $n \in \mathbb{N}$.
\end{lemma}

\begin{proof}
	We assert that the set
	\[
		\Omega := \{\ell \in \mathbb{N} : \{k \in \mathbb{N}_{>\ell} : s_k N \subseteq s_{\ell} N \} \ \text{ is infinite} \}
	\]
	is nonempty.  Since $(\bigcup_{i=1}^r s_iN)_{r \ge 1}$ is an ascending chain of $s$-ideals, there is some $m \in \mathbb{N}$ such that $\bigcup_{n\in\mathbb{N}} s_nN=\bigcup_{i=1}^m s_iN$. Consequently, there is some $j \in \ldb 1,m \rdb$ such that $\{k \in \mathbb{N} : s_k N \subseteq s_j N \}$ is infinite. Therefore $\{k \in \mathbb{N}_{>j} : s_kN\subseteq s_jN\}$ is infinite, whence $j \in \Omega$. Thus, $\Omega$ is nonempty.
	
	Now fix $x \in \Omega$. Next we assert that there is some $y\in\Omega$ such that $y>x$ and $s_yN\subseteq s_xN$. Since  $\Sigma = \{k \in \mathbb{N}_{>x} : s_k N \subseteq s_x N \}$ is infinite, there is some injective map $\psi \colon \mathbb{N} \rightarrow \Sigma$. As $(\bigcup_{i=1}^r s_{\psi(i)}N)_{r \ge 1}$ is an ascending chain of $s$-ideals, there is some $m \in \mathbb{N}$ such that $\bigcup_{n \in \mathbb{N}} s_{\psi(n)}N = \bigcup_{i=1}^m s_{\psi(i)}N$. We infer that there is some $j \in \ldb 1,m \rdb$ such that $\{k \in \mathbb{N} : s_{\psi(k)} N \subseteq s_{\psi(j)} N \}$ is infinite. Set $y = \psi(j)$. Since $\psi$ is injective, we obtain that $\{k \in \mathbb{N}_{>y} : s_k N \subseteq s_y N \}$ is infinite. Therefore $y \in \Omega$. Since $y = \psi(j) \in \Sigma$, we see that $y>x$ and $s_y N \subseteq s_xN$.
	
	Let $f \colon \Omega \rightarrow \Omega$ be defined by $f(x) = \min\{y \in \Omega : y>x \ \text{and} \ s_y N \subseteq s_x N \}$ for each $x \in \Omega$. Then we recursively define a map $\varphi \colon \mathbb{N} \rightarrow \Omega$ by setting $\varphi(1) = \min \Omega$ and $\varphi(n+1) = f(\varphi(n))$ for each $n \in \mathbb{N}$. Hence we obtain a map $\varphi \colon \mathbb{N} \rightarrow \mathbb{N}$ such that $\varphi(n+1) > \varphi(n)$ and $s_{\varphi(n+1)} N \subseteq s_{\varphi(n)} N$ for every $n \in \mathbb{N}$.
\end{proof}

\smallskip
\begin{lemma}\label{2.4}
	Let $M$ be a monoid, and let $M_0$ and $N$ be submonoids of $M$ such that $N$ is $s$-noetherian and $M = M_0 N$. Then $M$ satisfies the {\rm ACCP} if and only if  each ascending chain $(r_n M)_{n \ge 1}$ stabilizes when $r_n \in M_0$ for every $n \in \nn$.
\end{lemma}

\begin{proof}
	Obviously, it suffices to verify the reverse implication. Let $(x_nM)_{n \ge 1}$ be an ascending chain of principal ideals in $M$. Take a sequence $(r_n)_{n \ge 1}$ of elements of $M_0$ and a sequence $(s_n)_{n \ge 1}$ of elements of~$N$ such that $x_n = r_n s_n$ for every $n \in \nn$. By virtue of Lemma~\ref{2.3}, we can take a map $\varphi \colon \mathbb{N} \rightarrow \mathbb{N}$ such that $\varphi(n+1) > \varphi(n)$ and $s_{\varphi(n+1)}N \subseteq s_{\varphi(n)} N$ for every $n \in \mathbb{N}$. For each $n \in \mathbb{N}$,
	\[
		r_{\varphi(n)} = \frac{x_{\varphi(n)}}{s_{\varphi(n)}} = r_{\varphi(n+1)}\frac{x_{\varphi(n)}s_{\varphi(n+1)}}{x_{\varphi(n+1)}s_{\varphi(n)}} \in r_{\varphi(n+1)}M,
	\]
	whence $r_{\varphi(n)} M \subseteq r_{\varphi(n+1)} M$. Consequently, there is some $\ell \in \mathbb{N}$ such that $r_{\varphi(k)}M = r_{\varphi(\ell)}M$ for every $k \in \mathbb{N}_{\geq \ell}$. If $k \in \mathbb{N}_{\geq\ell}$, then $s_{\varphi(k)}N\subseteq s_{\varphi(\ell)}N$, and so
	\[
		x_{\varphi(k)}M=r_{\varphi(k)}s_{\varphi(k)}M=r_{\varphi(\ell)}s_{\varphi(k)}M\subseteq r_{\varphi(\ell)}s_{\varphi(\ell)}M=x_{\varphi(\ell)}M\subseteq x_{\varphi(k)}M.
	\]
	After setting $m := \varphi(\ell)$, we infer that $x_{\varphi(k)}M = x_mM$ for every $k \ge \ell$ and, therefore, $x_kM = x_mM$ for every $k \geq m$.
\end{proof}

\bigskip
%%%%%%%%%%%%%%%%%%%%%%%%%%%%%%%%%%%%%%
%%%%%%%%%%%%%%%%%%%%%%%%%%%%%%%%%%%%%%
\section{A Realization Result for a Single Set $L \subseteq \nn_{\ge 2}$}  \label{sec:overatomicity}

In this section, we construct, for each nonempty subset $L \subseteq \mathbb N_{\ge 2}$, a Puiseux monoid $M$ that satisfies the ACCP and has $L$ as a length set (Theorem \ref{3.5}). The monoid $M$ will be constructed as the union of an ascending sequence $(M_n)_{n \ge 1}$ of finitely generated Puiseux monoids. The starting point of this construction is established by Propositions \ref{prop:prime reciprocal} and \ref{realization-finite-set}.

We start with  Puiseux monoids of the form $\big\langle \frac1p : p \in P \big\rangle$ for some infinite subset $P \subseteq \mathbb N$.  This class of monoids   pops up at various places in the literature, oftentimes in the special case where $P$ is the set of primes (see \cite[Example~2.1]{AAZ90} or \cite[Proposition~4.2]{Go22a}). The first three properties in Proposition~\ref{prop:prime reciprocal} are known; they were first established in \cite[Proposition~3.5 and Theorem~4.5]{AGH21}. For convenience of the reader, we include their short proofs here. Moreover, we also completely determine the full system of length sets of the corresponding monoids, and this is the first class of monoids satisfying the ACCP but not the bounded factorization property for which this is done (the Puiseux monoids in \cite[Theorem~3.3]{CGG20a} do not satisfy the ACCP). Even in the case of BF-monoids, the full system of length sets has so far been completely determined only in very special cases (see \cite{Ge-Sc-Zh17b} for an overview of results).

\smallskip
\begin{prop} \label{prop:prime reciprocal}
	Let $P \subseteq \mathbb N_{\ge 2}$ be an infinite set of pairwise relatively prime positive integers, and let~$M_P$ be the Puiseux monoid generated by the set $\big\{ \frac1p : p \in P \big\}$. Then the following statements hold.
	\begin{enumerate}
		\item  $\mathcal{A}(M_P) = \big\{ \frac1p : p \in P \big\}$.
		\smallskip
		
		\item  Every  $q \in M_P$ can be uniquely written in the form
	     \begin{equation} \label{eq:canonical dec of prime reciprocal}
			  q = N(q) + \sum_{p \in P} c_p \frac1p,
		 \end{equation}
		 for some $N(q) \in \nn_0$ and coefficients $c_p \in \ldb 0, p - 1 \rdb$ for every $p \in P$ such that almost all $c_p$ are equal to zero. In addition, for any $q, q' \in M_P$ the following statements hold:
		  \begin{enumerate}
			   \item $N(q) \le N(q')$ if $q \mid_{M_P} q'$, and
			   \smallskip
							
			   \item $q - N(q) \in M_P$ has  unique factorization.
		  \end{enumerate}
		  \smallskip
		
		  \item $M_P$ satisfies the {\rm ACCP}.
	      \smallskip
		
		  \item $\mathcal{L}(M_P) = \{m + n P : m,n \in \nn_0\}$.
	\end{enumerate}
\end{prop}

\begin{proof}
	We set $A := \big\{ \frac1p : p \in P \big\}$.

	(1) Since $M_P$ is reduced, it suffices to verify that $A$ is a minimal generating set of $M_P$ (see \cite[Proposition~1.1.7]{Ge-HK06a}). Assume to the contrary that there is $p_0 \in P$ such that $A \setminus \{p_0\}$ is a generating set. Then there are $p_1, \ldots, p_n \in P \setminus \{p_0\}$ and $c_1, \ldots, c_n \in \mathbb N$ such that
	\[
		\frac{1}{p_0} = c_1 \frac{1}{p_1} + \dots + c_n\frac{1}{p_n}.
	\]
	If $m = p_1  \cdots  p_n$, then
	\[
		\frac{m}{p_0} = c_1 \frac{m}{p_1} + \dots + c_n\frac{m}{p_n} \in \mathbb N,
	\]
	whence $p_0 \mid m$, a contradiction.
	\smallskip
	
	(2) Since $A$ is a generating set, every $q \in M_P$ can be written in the indicated form. The uniqueness follows from an argument as in part~(1). Then statements (a) and (b) are clear.
	\smallskip

	(3) Let $(q_n + M_P)_{n \ge 1}$ be an ascending chain of principal ideals of $M_P$. It follows (a) in part~(2) that the sequence $(N(q_n))_{n \ge 1}$ must stabilize, say $N(q_n) = N(q_m)$ for some $m \in \mathbb N$ and all $n \ge m$. Thus, by (b) in part~(2), the element $q_n- N(q_n)$ has unique  factorization for every $n \ge m$. This immediately implies that the ascending chain $(q_n + M_P)_{n \ge 1}$ must stabilize.
	\smallskip

	(4) We proceed in two steps.
	\smallskip
	
	\noindent (i) Take $q \in M_P$ and write $q = N(q) + \sum_{p \in P} c_p \frac1p$ as in \eqref{eq:canonical dec of prime reciprocal}. We claim that
	\begin{equation} \label{eq:set of length formula for PR}
		\mathsf{L}(q) = s_q + N(q) P,
	\end{equation}
	where $s_q := \sum_{p \in P} c_p$. To verify that $\mathsf{L}(q) \subseteq s_q + N(q) P$, suppose that $z = \sum_{p \in P} c'_p \frac1p \in \mathsf Z (q) \subseteq \mathsf Z (M_P)$ is a factorization of~$q$ with all $c_p' \in \mathbb N_0$ and almost all of them equal to zero. For each $p \in P$, we can write $c'_p = m_p p + r_p$ for some $m_p, r_p \in \nn_0$ with $r_p \in \ldb 0, p-1 \rdb$. Then
	\[
		q = \sum_{p \in P} c'_p \frac1p = \sum_{p \in P} m_p + \sum_{p \in P} r_p \frac1p
	\]
	and, by the uniqueness of the decomposition in~\eqref{eq:canonical dec of prime reciprocal}, it follows that $N(q) = \sum_{p \in P} m_p$ and $c_p = r_p$ for every $p \in P$, which implies that $\sum_{p \in P} r_p = s_q$. Thus, we obtain that
	\[
		|z| = \sum_{p \in P} c'_p = \sum_{p \in P} r_p + \sum_{p \in P} m_p p \in s_q +  N(q) P,
	\]
	and hence the inclusion $\mathsf{L}(q) \subseteq s_q + N(q) P$ holds. For the reverse inclusion, take $\ell \in s_q + N(q) P$ and write $\ell = s_q + (p_1 + \dots + p_{N(q)})$ for some $p_1, \dots, p_{N(q)} \in P$. Now consider the factorization
	\[
		w = \sum_{p \in P} (t_p p + c_p) \frac1p \ \in \mathsf Z (M_P),
	\]
	where $t_p = |\{i \in \ldb 1, N(q) \rdb : p_i = p\}|$ for every $p \in P$. Since
	\[
		\sum_{p \in P} (t_p p + c_p) \frac1p = \sum_{p \in P} t_p + \sum_{p \in P} c_p \frac1p = N(q) + \sum_{p \in P} c_p \frac1p = q,
	\]
	it follows that $w$ is a factorization of $q$. In addition,
	\[
		|w| = \sum_{p \in P}(c_p + t_p p) = \sum_{p \in P} c_p +  \sum_{p \in P} t_pp = s_q + \sum_{i=1}^{N(q)} p_i = \ell,
	\]
	whence $\ell \in \mathsf{L}(q)$. Thus, $\mathsf{L}(q) = s_q + N(q) P$, which implies that $\mathcal{L}(M_P) \subseteq \{m + n P : m,n \in \nn_0\}$.
	\smallskip
	
	\noindent (ii) For the reverse inclusion, let $m,n \in \nn_0$ be given. Now take $p_1, \ldots, p_n \in P$ such that $p_1 < \dots < p_n$, and consider the rational number
	\[
		q := m + \sum_{i=1}^n \frac1{p_i} \ \in M_P.
	\]
	Since $m + \sum_{i=1}^n \frac1{p_i}$ is the unique decomposition of $q$ satisfying the restrictions of~\eqref{eq:canonical dec of prime reciprocal}, it follows from~\eqref{eq:set of length formula for PR} that $m + n P = \mathsf{L}(q) \in \mathcal{L}(M_P)$. As a consequence, the inclusion $ \{m + n P : m,n \in \nn_0\} \subseteq \mathcal{L}(M_P)$ also holds, which completes our argument.
\end{proof}

The following proposition is also needed to argue the primary result of this section.

\begin{prop} \label{realization-finite-set}
For every finite, nonempty subset $L \subseteq \mathbb N_{\ge 2}$, there exist a numerical monoid $M$ and a squarefree element $m \in M$ such that $\mathsf L_M (m) = L$.
\end{prop}

\begin{proof}
	See \cite[Theorem 3.3]{Ge-Sc18e}.
\end{proof}

\smallskip
We consider a finite-dimensional real vector space $V = \mathbb R^d$ with $d \ge 1$ endowed with the following lattice order: if $\boldsymbol{x} = (x_1, \dots, x_d)$ and $\boldsymbol{y} = (y_1, \dots, y_d)$ are vectors in~$V$, then $\boldsymbol x \le \boldsymbol y$ if $x_i \le y_i$ for all $i \in \ldb 1,d \rdb$. Then $V^+ = \{ \boldsymbol x \in V : \boldsymbol x \ge 0 \} = \mathbb R_{\ge 0}^d$. Let $\|\boldsymbol x\| = |x_1| + \dots + |x_d|$ denote the $1$-norm. An additive submonoid $M \subseteq V$ is called {\it positive} if $M \subseteq V^+$.

\smallskip
\begin{prop} \label{prop:ACCP}
	Let  $(V=\mathbb R^d, \|\cdot \|, \le)$ with $d\ge 1$ and with norm and order as above. Let $M \subseteq V$ be a positive submonoid, and let $M_0$ and $N_0$ be  submonoids of $M$ such that $M = M_0 + N_0$. If $M_0$ satisfies the {\rm ACCP} and $N_0$ is $s$-noetherian, then $M$ satisfies the {\rm ACCP}.
\end{prop}

\begin{proof}
	If $N_0 =\{0\}$, then $M=M_0$ satisfies the ACCP. Suppose that $N_0 \ne \{0\}$. Since $N_0$ is reduced and $s$-noetherian,  $\mathcal{A}(N_0)$ is finite and nonempty. We set $\alpha := \min \{ \|a\| : a \in \mathcal{A}(N_0)\}$ and observe that $\alpha = \min \{ \|b\| : b \in N_0 \setminus \{0\} \} >0$.
	
	Let $(r_n)_{n \ge 1}$ be a sequence of elements of $M_0$ such that $r_n + M \subseteq r_{n+1} + M$ for each $n \in \mathbb{N}$. Observe that $r_n - r_{n+1} \in M \subseteq V^+$ and so $r_n \geq r_{n+1} \geq 0$ for every $n \in \mathbb{N}$. Therefore the sequence $(\|r_n\|)_{n \ge 1}$ is monotonously decreasing and bounded from below, and thus it is convergent.

	\smallskip
	For every $n \in \mathbb N$, we set $r_n := (r_n^{(1)}, \ldots, r_n^{(d)})$. Since $r_n \ge r_{n+1}$, it follows that
	\[
		r_n^{(i)} \ge r_{n+1}^{(i)} \quad \text{for every} \quad i \in \ldb 1,d \rdb.
	\]
	Since $(\|r_n\|)_{n \ge 1}$ is convergent, there is some $\beta \in \mathbb R_{\ge 0}$ such that
	\[
		\lim_{n \to \infty} \|r_n\| = \lim_{n \to \infty} (r_n^{(1)} + \dots +  r_n^{(d)}) = \beta.
	\]
	This implies that
	\[
		\begin{aligned}
			\lim_{n \to \infty} \| r_n-r_{n+1} \| & = \lim_{n \to \infty} \big( (r_n^{(1)}-r_{n+1}^{(1)})  + \dots +  (r_n^{(d)}- r_{n+1}^{(d)}) \big) \\
			& = \lim_{n \to \infty} \big(  (r_n^{(1)} + \dots +  r_n^{(d)}) -  (r_{n+1}^{(1)} + \dots +  r_{n+1}^{(d)}) \big) \\
			 & = \lim_{n \to \infty} (r_n^{(1)} + \dots +  r_n^{(d)}) - \lim_{n \to \infty} (r_{n+1}^{(1)} + \dots +  r_{n+1}^{(d)}) = 0.
		\end{aligned}
	\]
	Thus, there is some $\ell\in\mathbb{N}$ such that $\| r_n-r_{n+1} \| < \alpha $ for each $n\in\mathbb{N}_{\geq\ell}$. Because $M = M_0 + N_0$, for each $n \in \mathbb{N}_{\geq\ell}$ we can pick some $x_n \in M_0$ and $y_n \in N_0$ such that $r_n - r_{n+1} = x_n + y_n$. Now for each $n \in \nn_{\ge \ell}$, the fact that $y_n\in N_0$, together with
	\[
		0 \leq \|y_n \| = \| r_n-r_{n+1}-x_n \| \leq \|r_n-r_{n+1}\| < \alpha,
	\]
	guarantees that $y_n=0$, and so $r_n = r_{n+1} + x_n$. This implies that $r_n + M_0 \subseteq r_{n+1} + M_0$ for every $n \in \mathbb{N}_{\geq\ell}$, and thus we can take $m \in \mathbb{N}_{\geq\ell}$ such that $r_n+M_0 = r_m+M_0$ for each $n\in\mathbb{N}_{\geq m}$. Consequently, $r_n+M=r_m+M$ for each $n\in\mathbb{N}_{\geq m}$. It follows now from Lemma~\ref{2.4} that $M$ satisfies the ACCP.
\end{proof}

The condition that $N_0$ is $s$-noetherian in Proposition~\ref{prop:ACCP} is crucial. The following example illustrates this observation.

\begin{example} \label{3.4}
	Consider the Puiseux monoids $M_0 := \{0\} \cup \qq_{\ge 1}$ and $N_0 = \big\langle \frac1p : p \in \pp \big\rangle$. Since $0$ is not a limit point of $M_0^\bullet$, it follows from \cite[Proposition~4.5]{fG19} that $M_0$ is a BF-monoid and, therefore, it satisfies the ACCP. On the other hand, it follows from part~(3) of Proposition~\ref{prop:prime reciprocal} that $N_0$ also satisfies the ACCP. However, let us argue that the monoid  $M := M_0 + N_0$ is not even atomic, and so it cannot satisfy the ACCP. Since $\min M_0^{\bullet} = 1$, none of the non-zero elements of $M_0$ divides in $M$ any of the elements~$\frac1p$ with $p \in \pp$. Therefore it follows from part~(1) of Proposition~\ref{prop:prime reciprocal} that $\big\{ \frac1p : p \in \pp \big\} \subseteq \mathcal{A}(M)$. In addition, no element of $\qq_{\ge 1}$ can be an atom of $M$: indeed, $1 \notin \mathcal{A}(M)$ because $1$ is the sum of $p$ copies of the atom~$\frac1p$ for any $p \in \pp$, while any $q \in \qq$ with $q > 1$ can be written as $q = \frac1p + \big(q - \frac1p \big)$ in $M$ for some $p \in \pp$ provided that $p$ is sufficiently large so that $\frac1p < q-1$. Hence $\mathcal{A}(M) = \big\{ \frac1p : p \in \pp \big\}$. Now we can see that~$M$ is not atomic as, for instance, the element $\frac54$ cannot be written as a sum of atoms in $M$.
\end{example}

The next lemma is the last crucial tool we need in order to establish the main result of this section.

\begin{lemma} \label{lem:auxiliary}
Let $M_0$ be a Puiseux monoid satisfying the {\rm ACCP} with $1 \in M_0 \setminus \mathcal{A}(M_0)$, and let $N$ be a numerical monoid containing a prime $p$ such that
\[	
	p > \max \mathcal{A}(N) \quad \text{and} \quad  \mathsf v_p(M_0^\bullet) \subseteq \nn_0.
\]
Then the Puiseux monoid $M:= M_0 + p^{-1}N$ has the following properties.
\begin{enumerate}
	\item[(1)] $M$ satisfies the {\rm ACCP} with $\mathcal{A}(M) = \mathcal{A}(M_0) \cup \mathcal{A}\big(p^{-1} N \big)$.
	\smallskip
			
	\item[(2)] $\mathsf{L}_M(1) = \mathsf{L}_{M_0}(1) \cup \mathsf{L}_N(p)$.
\end{enumerate}
\end{lemma}

\begin{proof}
	Clearly,  $N_0 := p^{-1} N$ is isomorphic to $N$, whence   $N_0$ is a finitely generated Puiseux monoid with $1 \in N_0$. Since $p > \max \mathcal{A}(N)$, we infer that  $\max \mathcal A (N_0) < 1$.
	\smallskip
	
	(1) By Proposition \ref{prop:ACCP},  the monoid $M = M_0 + N_0$ satisfies the ACCP, whence it is atomic.  Since both $M_0$ and $N_0$ are atomic monoids, $\mathcal{A}(M_0) \cup \mathcal{A}(N_0)$ is a generating set of $M$. Since in a reduced monoid the set of atoms is the smallest generating set (see \cite[Proposition 1.1.7]{Ge-HK06a}), it follows that $\mathcal{A}(M) \subseteq \mathcal{A}(M_0) \cup \mathcal{A}(N_0)$. For the reverse inclusion, take $a_0 \in \mathcal{A}(M_0) \cup \mathcal{A}(N_0)$ and write $a_0=b_0+c_0$ with $b_0, c_0 \in M$. Since $\mathcal{A}(M_0) \cup \mathcal{A}(N_0)$ is a generating set of $M$, we obtain a representation
	\begin{equation} \label{eq:aux}
		a_0 = \sum_{a \in \mathcal{A}(M_0)} c_a a + \sum_{a \in \mathcal{A}(N_0)} d_a a
	\end{equation}
	with non-negative integers $c_a, d_a$ and almost all $c_a$ and $d_a$ equal to zero.

	Assume first that $a_0 \in \mathcal{A}(M_0)$. Since $\mathsf v_p(M_0^\bullet) \subseteq \nn_0$, we infer that  $m := \sum_{a \in \mathcal{A}(N_0)} d_a a \in \nn_0$. Now as $1 \in M_0 \setminus \mathcal{A}(M_0)$, it follows that $m=0$. Thus, based on~\eqref{eq:aux} and the fact that $a_0 \in \mathcal{A}(M_0)$, we see obtain that $c_{a_0} = 1$ and so $c_a = 0$ for all $a \in \mathcal{A}(M_0) \setminus \{a_0\}$. Thus, $\mathcal{A}(M_0) \subseteq \mathcal{A}(M)$.

	Assume now that $a_0 \in \mathcal{A}(N_0)$. Again, since $\mathsf v_p(M_0^\bullet) \subseteq \nn_0$, we see that $n := a_0 - \sum_{a \in \mathcal{A}(N_0)} d_a a \in \nn_0$. Thus, from the inequality $a_0 \le \max \mathcal A (N_0) < 1$, we infer that $n=0$, whence $a_0 =  \sum_{a \in \mathcal{A}(N_0)} d_a a$. This implies that $c_a = 0$ for all $a \in \mathcal{A}(M_0)$. Now, since~$N_0$ is atomic, $d_{a_0} = 1$ and, therefore, $d_a = 0$ for every $a \in \mathcal{A}(N_0) \setminus \{a_0\}$, whence $a_0 \in \mathcal{A}(M)$.
	\smallskip
	
	(2) Observe that $\mathsf{L}_{M_0}(1) \cup \mathsf{L}_N(p)  = \mathsf{L}_{M_0}(1) \cup \mathsf{L}_{N_0}(1) \subseteq \mathsf{L}_{M}(1)$, where the last inclusion holds because $\mathcal{A}(M_0) \cup \mathcal{A}(N_0) \subseteq \mathcal{A}(M)$.  To verify the reverse inclusion, take $\ell \in \mathsf{L}_M(1)$ and write
	\begin{equation} \label{eq:aux I}
		1 = \sum_{a \in \mathcal{A}(M_0)} c_a a + \sum_{a \in \mathcal{A}(N_0)} d_a a,
	\end{equation}
	where all $c_a, d_a \in \mathbb N_0$ and almost all of them equal to zero, such that
	\[
		\ell = \sum_{a \in \mathcal{A}(M_0)} c_a + \sum_{a \in \mathcal{A}(N_0)} d_a.
	\]
	After applying the $p$-adic valuation map on both sides of the equality~\eqref{eq:aux I}, as we did in part~(1) with the equality~\eqref{eq:aux}, we obtain that $d := 1 - \sum_{a \in \mathcal{A}(N_0)} d_a a \in \{0,1\}$. If $d=0$, then we can deduce from~\eqref{eq:aux I} that $c_a = 0$ for all $a \in \mathcal{A}(M_0)$, whence $\ell = \sum_{a \in \mathcal{A}(N_0)} d_a \in \mathsf{L}_{N_0}(1) = \mathsf{L}_N(p)$. If $d=1$, then we see that $d_a = 0$ for every $a \in \mathcal{A}(N_0)$, whence $\ell = \sum_{a \in \mathcal{A}(M_0)} c_a \in \mathsf{L}_{M_0}(1)$.
\end{proof}

We are in a position to prove the main result of this section.

\begin{theorem} \label{3.5}
	For every nonempty subset $L \subseteq \nn_{\ge 2}$, there exists a Puiseux monoid~$M$ satisfying the {\rm ACCP} such that $1 \in M$ and $\mathsf{L}_M(1) = L$.
\end{theorem}

\begin{proof}
	Let $\emptyset \ne L \subseteq \nn_{\ge 2}$ be given. If $L$ is finite, then it follows from Proposition~\ref{realization-finite-set} that there exists a numerical monoid $N$ containing a non-zero element $m$ such that $\mathsf{L}_N(m) = L$. Then $M := m^{-1} N$  a Puiseux monoid isomorphic to $N$ with $\mathsf{L}_M(1) = \mathsf{L}_N(m) = L$.
		
	Now assume that $L$ is infinite. Fix $\ell_0 \in L$ such that $\ell_0 \ge 10$ and set $L_0 := L_{\le \ell_0}$. Let $(\ell_n)_{n \ge 1}$ be the strictly increasing sequence whose underlying set is $L \setminus L_0$ and, for each $n \in \nn$, set
	\[
		L_n := L_0 \cup \{\ell_1, \dots, \ell_n\}.
	\]
	By Proposition \ref{realization-finite-set}, there are a numerical monoid $N_0$ and some $m_0 \in N_0$ such that $\mathsf{L}_{N_0}(m_0) = L_0$.
	Let $M_{\mathbb P}$ be the Puiseux monoid generated by the set $A := \big\{ \frac1p : p \in \mathbb P \big\}$.
	By Proposition~\ref{prop:prime reciprocal}, the monoid $M_{\mathbb P}$ satisfies the ACCP and $A$ is the set of atoms of $M_{\mathbb P}$. Thus, by Proposition~\ref{prop:ACCP}, the Puiseux monoid $m_0^{-1} N_0 + M_{\mathbb P}$ satisfies the ACCP. We continue with the following claim.
	\smallskip
	
	\noindent \textsc{\bf Claim.} There exists an ascending chain $(M_n)_{n \ge 0}$ of finitely generated Puiseux monoids $M_n$ such that for every $n \ge 0$, the following three conditions hold.
	\begin{enumerate}
		\item $M_n \subseteq m_0^{-1} N_0 + M_{\mathbb P}$;
		\smallskip
		
		\item $1 \in M_n$ and $\mathsf{L}_{M_n}(1) = L_n$;
		\smallskip
		
		\item $\mathcal A (M_{n-1}) \subseteq \mathcal A (M_{n})$ for $n \ge 1$.
	\end{enumerate}
	
	\noindent
	\begin{proof}[Proof of Claim] 
		We proceed by induction on $n$. If $n=0$, then  $M_0 := m_0^{-1} N_0$ satisfies the required properties.
			
		Now suppose that for some $n \in \nn_0$ we have already constructed finitely generated Puiseux monoids $M_0, \dots, M_{n-1}$ with $M_0 \subseteq \dots \subseteq M_{n-1}$ such that conditions~(1)--(3) are satisfied.  To construct $M_n$, we choose a prime $p$  such that $p > 4\ell_n^2$ and $\mathsf v_p(M_{n-1}) \subseteq \nn_0$, that is $p \nmid \mathsf{d}(q)$ for any $q \in M_{n-1}^\bullet$ (this is possible because $M_{n-1}$ is finitely generated, and so $\mathsf{d}(M_{n-1}^\bullet)$ is a finite set). Now take $s, t \in \nn_{\ge 2}$ such that $s + t = \ell_n$ and $\gcd(s, t) = 1$: this can be done by taking
		\begin{itemize}
			\item $\{s, t\} = \{j, j+1\}$ if $\ell_n = 2j+1$ for some $j \in \nn$;
			\smallskip
			
			\item $\{s, t\} = \{2j - 1, 2j+1 \}$ if $\ell_n = 4j$ for some $j \in \nn$;
			\smallskip
			
			\item $\{s, t\} = \{2j-1, 2j+3 \}$ if $\ell_n = 4j+2$ for some $j \in \nn$.
		\end{itemize}
		Observe that for any of the choices of the set $\{s,t\}$ above, the fact that $\ell_n \ge 11$ implies that $\min\{s,t\} \ge 5$, the equality taking place when $\ell_n \in \{11,12,13\}$. Now set $m := \max\{s, t\}$, and note that the inequality $\ell_n \ge 11$ implies that $\min\{2s, 2t\} > m$. Since $\gcd(s, t) = 1$, we obtain that $S := \langle s, t \rangle$ is a numerical monoid with Frobenius number $\mathsf{f}(S) = s t - s - t$. In addition, the equality $\ell_n = s + t$ implies that $\ell_n^2 > s^2 + t^2$, which in turn implies that
		\[
			p > \ell_n^2 > s^2 + t^2 > st > \mathsf{f}(S).
		\]
		Therefore there exist $b, c \in \nn$ such that $b s + c t = p$. Since $\gcd(s, t) = 1$, the set of solutions $(x,y)$ of the Diophantine equation $x s + y t = p$ is $\{(b - jt, c + js) : j \in \zz\}$. If $b \le m$, then $b + 2t > m$ (because $\min\{2s, 2t\} > m$) and, therefore,
		\[
			c - 2s = \frac{1}{t}\big( p - b s - 2st \big) > \frac{1}{t}\big( p - 3 \ell_n^2 \big) > \frac{\ell_n^2}{t} > m,
		\]
		where the first inequality follows from $\max\{b,s,t\} < \ell_n$, the second inequality follows from $p > 4 \ell_n^2$, and the third inequality follows from $\ell_n > \max\{m,t\}$. Thus, if $b \le m$, then after replacing $(b, c)$ by another suitable solution of $x s + y t = p$ (namely, $(b + 2t, c - 2s)$), one can further assume that $\min \{b, c\} > m$. On the other hand, if $c \le m$, then $c + 2s > m$, and one can similarly obtain that
		\[
			b - 2t = \frac{1}{s}\big( p - ct - 2st \big) > \frac{1}{s}\big( p - 3 \ell_n^2 \big) > \frac{\ell_n^2}{s} > m.
		\]
		 Thus, if $c \le m$, then we can replace $(b, c)$ by the pair $(b - 2t, c + 2s)$ and further assume that the inequality $\min \{b, c\} > m$ holds. Hence we can assume, without loss of generality, that $\min\{b, c\} > m$.
		
		Consider the numerical monoid
		\[
			N := \langle b, c \rangle,
		\]
		and observe that $p = s b + t c \in N$. We assert that $p$ has  unique factorization in $N$. To argue this, first note that the equality $s b + t c = p$, along with the fact that $\gcd(b,c) = 1$, guarantees that any factorization of $p$ in $N$ must have the form $(s - jc) b + (t + j b) c$ for some $j \in \zz$. However, observe that $s - j c \le s - c < 0$ when $j \ge 1$, while $t + j b \le t - b < 0$ when $j \le -1$ (both statements follow from the previous assumption that $\min\{b, c\} > m$). As a consequence, the only factorization of $p$ in $N$ is $s b + t c$, and so $\mathsf{L}_N(p) = \{s + t\} = \{ \ell_n \}$.
		
		Now we define  the Puiseux monoid $M_n$ as
		\begin{equation} \label{construction}
			M_n := M_{n-1} + p^{-1} N.
		\end{equation}
		By the induction hypothesis, $M_{n-1} \subseteq m_0^{-1} N_0 + M_{\mathbb P}$. This, together with the fact that $p^{-1} N$ is a submonoid of $M_{\mathbb P}$, ensures that $M_n = M_{n-1} + p^{-1} N  \subseteq m_0^{-1} N_0 + (M_{\mathbb P} + p^{-1} N) = m_0^{-1} N_0 + M_{\mathbb P}$, which yields the condition~(1) above. In addition, it follows from part~(2) of Lemma~\ref{lem:auxiliary} that
		\[
			\mathsf{L}_{M_n}(1) = \mathsf{L}_{M_{n-1}}(1) \cup \mathsf{L}_N(p) = L_{n-1} \cup \{\ell_n\} = L_n.
		\]
		Therefore $M_n$ also satisfies condition~(2) above, and it satisfies condition~(3) by Lemma \ref{lem:auxiliary}. This completes the inductive step of our construction.
		\qedhere \, (Proof of the Claim)
	\end{proof}
	\smallskip
	
	Let $(M_n)_{n \ge 0}$ be as in the established claim and consider the Puiseux monoid
	\[
		M := \bigcup_{n \ge 0} M_n \ \subseteq \ m_0^{-1} N_0 + M_{\mathbb P}.
	\]
	Since $m_0^{-1} N_0 + M_{\mathbb P}$ is a reduced monoid satisfying the ACCP,  $M$ satisfies the  ACCP by part~(1) of Lemma~\ref{2.2}.
	Furthermore, Lemma \ref{lem:atoms of the union} implies that  $\mathcal{A}(M) \subseteq \bigcup_{n \ge 0} \mathcal{A}(M_n)$. In order to argue the reverse inclusion, take $a \in \bigcup_{n \ge 0} \mathcal{A}(M_n)$ and then pick $j \in \nn_0$ such that $a \in \mathcal{A}(M_j)$. Now write $a = b+c$ for some $b,c \in M$, and take $k \in \nn_{\ge j}$ large enough so that $b,c \in M_k$. Since $(\mathcal{A}(M_n))_{n \ge 0}$ is an ascending chain, the fact that $k \ge j$ implies that $a \in \mathcal{A}(M_k)$. Therefore either $b=0$ or $c=0$, and so $a \in \mathcal A (M)$.
	
	We finally show that $\mathsf{L}_M(1) = L$. Since $\mathcal{A}(M) = \bigcup_{n \ge 0} \mathcal{A}(M_n)$, for each factorization $z$ in $\mathsf{Z}_M(1)$,  there exists $n \in \nn$ such that $z \in \mathsf{Z}_{M_n}(1)$ and, therefore, $|z| \in L_n \subseteq L$. Thus, $\mathsf{L}_M(1) \subseteq L$. To show the reverse inclusion, we choose some  $\ell \in L$. Since $L = \bigcup_{n \ge 0} L_n$, there is some $n \in \nn_0$ such that $\ell \in L_n = \mathsf{L}_{M_n}(1)$. Let $w$ be a factorization in $\mathsf{Z}_{M_n}(1)$ with $|w| = \ell$. Since the atoms of $M_n$ remain atoms in $M$, it follows that $w \in \mathsf{Z}_M(1)$, whence $\ell = |w| \in \mathsf{L}_{M}(1)$.
\end{proof}

In \cite[Proposition~3.1 and Theorems~3.3 and~3.4]{Ge-Go-Tr21}, a variety of properties of Puiseux monoids were considered. It turns out that Puiseux monoids satisfying the ACCP and having an infinite length set (as those in the statement of Theorem~\ref{3.5}) do not satisfy any of these properties. We briefly argue this observation.

\begin{remark} \label{seminormal}
	Let $M$ be a nontrivial Puiseux monoid. Then it follows from \cite[Proposition~3.1]{Ge-Go-Tr21} that $M$ is root closed if and only if $M$ is seminormal, which happens precisely when $M$ is a valuation monoid. Since atomic valuation monoids are discrete valuation monoids, whence factorial, a monoid that satisfies the hypothesis in the statement of Theorem~\ref{3.5} cannot be seminormal. Moreover, one can readily verify that all nontrivial Puiseux monoids are primary, and it follows from~\cite[Theorem~3.4]{Ge-Go-Tr21} that those that are strongly primary are also BF-monoids. As a consequence, if $M$ has an infinite length set, then $M$ is not strongly primary, the conductor $(M : \widehat M)$ to its complete integral closure $\widehat M$ is empty, and $0$ is a limit point of $M^{\bullet}$ (\cite[Theorem 3.4]{Ge-Go-Tr21}).
\end{remark}

\bigskip
%%%%%%%%%%%%%%%%%%%
%%%%%%%%%%%%%%%%%%%
\section{Proof of Theorem \ref{1.1}}   \label{4}
\label{sec:mainresult}

The main purpose of this last section is to prove our main result, Theorem~\ref{1.1}. To do so, we use not only the results we have established in the previous two sections but also the following proposition.

\begin{prop} \label{4.1}
	There exist a $\mathbb Q$-vector space $V$ and a positive submonoid $M \subseteq V$ satisfying the {\rm ACCP} such that, for each nonempty subset $L \subseteq \nn_{\ge 2}$, there is some $b_L \in M$ with $\mathsf L_M (b_L) = L$. Moreover, if $\min L > 2$, then there are uncountably many  elements $b \in M$ with $\mathsf L_M(b) = L$.
\end{prop}

\begin{proof}
	Let $\Omega$ be the set of all nonempty subsets of $\mathbb N_{\ge 2}$, and note that $\Omega$ is uncountable. In light of Theorem~\ref{3.5}, for each $L \in \Omega$, there exist a Puiseux monoid $M_L$ satisfying the ACCP and an element $b_L \in M_L$ such that $\mathsf L_{M_L} (b_L) = L$.  Then the monoid
	\[
		M := \bigoplus_{L \in \Omega} M_L \ \subseteq \  \bigoplus_{L \in \Omega} \mathbb Q = \mathbb Q^{(\Omega)} =: V
	\]
	is a positive submonoid of $V$, and it satisfies the ACCP by part~(2) of Lemma \ref{2.2}. Since $M_L \subseteq V$ is a divisor-closed submonoid of $M$, we infer that, for every $L \in \Omega$,
	\[
		L = \mathsf L_{M_L} (b_L) = \mathsf L_M (b_L).
	\]
	To argue the last statement, suppose that $\min L > 2$ and set $L' : = - 1 + L \subseteq \mathbb N_{\ge 2}$. Now choose an element $b_{L'} \in M_{L'}$ with $\mathsf L_{M} (b_{L'})=L'$, and then choose a subset $S \in \Omega \setminus \{L'\}$. Since $M_S$ is a nontrivial Puiseux monoid satisfying the ACCP, we can pick $a_S \in \mathcal{A}(M_S)$ and observe that
	\[
		\mathsf L_M (a_S + b_{L'}) = \mathsf L_M (a_S) + \mathsf L_M ( b_{L'}) = 1 + L' = L.
	\]
	Now the fact that $\Omega$ is an uncountable set allows us to choose $S$ in uncountably many ways, yielding uncountably many pairwise distinct element $a_S + b_{L'} \in M$ with length set $L$.
\end{proof}

\smallskip
The primary result of this paper, Theorem~\ref{1.1}, states that there exists a monoid algebra satisfying the ACCP and having its system of length sets as large as it can possibly be. Before proving that, we briefly revise some basics of monoid algebras (these and more can be found in~\cite{Gi84}). Let $D$ be a commutative ring with $D \ne 0$, and let $M$ be an additively written commutative semigroup. Then the elements   of the semigroup algebra $D[M]$ can  be written uniquely in the form
\[
	f = \sum_{m \in M} c_m X^m,
\]
where $c_m \in D$ for all $m \in M$ and $c_m=0$ for almost all $m \in M$; that is, the set $\{X^m : m \in M \}$ is a basis for $D[M]$ as a $D$-module. The following statements hold.
\begin{itemize}
	\item[(i)] The semigroup algebra $D[M]$ is an integral domain if and only if $D$ is an integral domain and~$M$ is a cancellative and torsion-free monoid (\cite[Theorem~8.1]{Gi84}).
	\smallskip
	
	\item[(ii)] Suppose that $D[M]$ is an integral domain and $M$ is reduced. Then $D[M]$ satisfies the ACCP if and only if  both $M$ and $D$ satisfy the ACCP (\cite[Theorem 13]{An-Ju15a}).
\end{itemize}

\smallskip
\begin{proof}[Proof of Theorem \ref{1.1}]
Let $D$ be an integral domain satisfying the ACCP, and let $M \subseteq V$ be as in Proposition~\ref{4.1}. By Properties (i) and (ii), the monoid algebra $D[M]$ satisfies the ACCP. Now observe that the assignments $m \mapsto X^m$ for every $m \in M$ induce a monoid isomorphism
\[
	\varphi \colon M \to H:= \{X^m : m \in M\},
\]
from the additive monoid $M$ to the multiplicative submonoid of $D[M]$ consisting of all monic monomials. By \cite[Theorem 11.1]{Gi84}, the equality $D[M]^{\times} = D^{\times}$ holds. Also, as $M$ is cancellative and torsion-free, every divisor of a nonzero monomial in $D[M]$ must be a nonzero monomial. Thus, the monoid $D^{\times}H$ is a divisor-closed submonoid of $D[M] \setminus \{0\}$.

Now let $L$ be a nonempty subset of $\mathbb N_{\ge 2}$. %be a nonempty subset with $|L|=1$ or with $\min L \ge 2$.
By Proposition \ref{4.1}, there exists some $b_L \in M$ such that $L = \mathsf L_M (b_L)$, whence
\[
	L = \mathsf L_M (b_L) = \mathsf L_H \big( \varphi (b_L) \big) = \mathsf L_{HD^{\times}} \big( \varphi (b_L) \big) = \mathsf L_{D[M]} \big( \varphi (b_L) \big).
\]
\end{proof}

\smallskip
We conclude this paper with two observations on the monoid algebra constructed in the given proof of Theorem~\ref{1.1}.

\begin{remark}
	(1) Let $D$ be an integral domain, and let $M$ be a torsion-free monoid. If the monoid algebra $D[M]$ has an infinite length set, then it cannot be a Mori domain, whence it can neither be a Noetherian domain nor a Krull domain. Recall that $D[M]$ is seminormal if and only if $D$ and $M$ are seminormal (\cite[Theorem~4.76]{Br-Gu09a}). Since the monoid $M$ constructed in Proposition~\ref{4.1} is not seminormal (see Remark~\ref{seminormal}), the monoid algebra $D[M]$ constructed in Theorem~\ref{1.1} is not seminormal, whence it is neither root closed nor integrally closed.
\smallskip

	(2) Let $L \subseteq \mathbb N_{\ge 2}$ be a nonempty subset with $\min L > 2$, and let $D[M]$ be as in the proof of Theorem~\ref{1.1}. As we did in the proof of Proposition~\ref{4.1}, we can choose uncountably many (non-associate) elements $b \in D[M]$ such that $\mathsf L_{D[M]}(b) = L$.
\end{remark}

\bigskip
%%%%%%%%%%%%%%%%
%%%%%%%%%%%%%%%%
\noindent {\bf Acknowledgment.} The authors thank Andreas Reinhart for fruitful discussions and for his help with Lemmas~\ref{2.3} and~\ref{2.4}.

\providecommand{\bysame}{\leavevmode\hbox to3em{\hrulefill}\thinspace}
\providecommand{\MR}{\relax\ifhmode\unskip\space\fi MR }
\providecommand{\MRhref}[2]{
  \href{http://www.ams.org/mathscinet-getitem?mr=#1}{#2}
}
\providecommand{\href}[2]{#2}

\end{document}